\documentclass{amsart}
\usepackage{amsmath,amsfonts} 
\usepackage{amsthm} 
\usepackage{amssymb} 
\usepackage[T1]{fontenc} 
\usepackage[utf8]{inputenc}
\usepackage{textcomp}
\usepackage[dvips]{graphicx}
\usepackage{verbatim}
\usepackage{todonotes}
\usepackage{mathrsfs}
\usepackage[english]{babel}
\usepackage[babel]{csquotes}
\usepackage{xy}
\xyoption{all}
\usepackage{hyperref}
\usepackage[style=alphabetic,hyperref,backend=bibtex]{biblatex}
\bibliography{bibliography}

\theoremstyle{definition}
\newtheorem{definition}{Definition}[section]
\theoremstyle{plain}
\newtheorem{theorem}[definition]{Theorem}

\newtheorem{proposition}[definition]{Proposition}
\newtheorem{lemma}[definition]{Lemma}
\newtheorem{corollary}[definition]{Corollary}
\theoremstyle{remark}

\DeclareMathOperator{\Spec}{Spec}

\DeclareMathOperator{\Gal}{Gal}

\begin{document}
\title[The normal hull and commutator group]{The normal hull and commutator group for nonconnected group schemes}
\author{Giulia Battiston}
\date{\today}
\address{ Ruprecht-Karls-Universität Heidelberg -Mathematisches Institut -
Im Neuenheimer Feld 288
D-69120 Heidelberg }
%\thanks{This work was supported by DFG-Forschergruppe 1920}
\subjclass[2010]{14L10, 20G07}
\email{gbattiston@mathi.uni-heidelberg.de}
\maketitle

\begin{abstract} We prove that there is a well behaved notion of normal hull for smooth algebraic group schemes over a field and that the commutator group $(G,H)$ is well defined for $H\subset G$ smooth, even when both of them are not connected.
\end{abstract}

If $G$ is an abstract group and $H\subset G$ is a subgroup, then the \emph{normal hull} $H^G$ of $H$ in $G$ can be equivalently defined as the smallest normal subgroup of $G$ containing $H$ and as the subgroup of $G$ generated by all conjugates of $H$ in $G$, that is 
\[H^G=\langle \bigcup_{g\in G} gHg^{-1}\rangle.\]
If now $G$ is a group scheme over a field $k$ and $H\subset G$ is a subgroup scheme, then the two notions can be generalized as follows: one one hand we denote again by $H^G$ the smallest normal subgroup scheme of $G$ containing $H$, on the other we can define the fppf subgroup-sheaf of $G$ given by
\[\tilde{H}^G(R)=\{l\in G(R)\mid \exists R\to R' \text{ fppf }, l \in H(R')^{G(R')}=\langle \bigcup_{g\in G(R')} gH(R')g^{-1}\rangle\}\]
for every $k$-algebra $R$.
We want to prove that the two notions coincide, that is $H^G(R)=\tilde{H}^G(R)$ for every $k$-algebra $R$, under the assumptions that $H$ and $G$ are smooth algebraic group schemes.

\subsection*{Notations} For $H,K$ two abstract subgroups of an abstract group $G$, we denote by $(H,K)$ their commutator, that the subgroup of $G$ spanned by $[h,k]=hkh^{-1}k^{-1}$ for $h\in H$ and $k\in K$. With the abbreviation fppf we mean faithfully flat and \emph{locally} of finite presentation.

All group schemes are over $\Spec k$ for some (non necessarily algebraically closed) field $k$ and by \emph{algebraic} we mean of finite type over $\Spec k$.

\subsection*{Acknowledgments} I would like to thank Wolfgang Soergel for very useful discussions and for pointing out to me the existence of Proposition~\ref{comm:conn}.

\section{A theorem of Schur}
In this section all groups are abstract groups.
\begin{lemma}\label{lem:norm}
Let $H$ be a subgoup of a group $G$. Then
\begin{itemize}
\item[i)] The group $(G,H)$ is normal in $G$;
\item[ii)] For every inner automorphism $\sigma$ of $G$, we have that $(G,H)=(G,\sigma(H))$
\item[iii)] The normal hull of $H$ in $G$ is equal to $H\cdot (G,H)$.
\end{itemize}
\end{lemma}
\begin{proof}
For every $g,c\in G$ and $h\in H$ we have that $c[g,h]c^{-1}=[cg,h][h,c]$, which proves $(i)$. If $\sigma$ is the inner automorphism induced by $c$, we have that $[g,chc^{-1}]=[c(c^{-1}gc)c^{-1},chc^{-1}]=c[c^{-1}gc,h]c^{-1}$ which is in $(G,H)$ as the latter is normal hence $(G,\sigma(H))\subset (G,H)$, hence the equality in $(ii)$.
As $(G,H)$ is normal, the group $H\cdot (G,H)$ is well defined. As every $[g,h]$ is contained in the normal hull of $H$, we have the inclusion $H\cdot (G,H)\subset H^G$, on the other hand for every $h\in H$ and $g\in G$ we have that $ghg^{-1}=[g,h]h^{-1}$ is in $(G,H)\cdot H=H\cdot (G,H)$, hence the equality.
\end{proof}

A classical theorem of Schur states that if the center of a group $G$ has finite index in $G$, then the group of commutators is finite. We need a slight generalization of this theorem:
\begin{proposition} If $H$ is a subgroup of $G$ such that its centralizer $C_G(H)$ has finite index, then $(G,H)$ is finite.
\end{proposition}
\begin{proof}
We follow the proof of Schur's theorem as presented in \cite[Prob.~5.21-5.24]{Dixon}. First notice that by Poincar\'e's theorem the normal core of $C_G(H)$ has finite index in $G$ as well, that we denote $n$. Hence for every $g\in G$ one has that $g^n\in C_G(H)$. In particular, if $g\in G$ and $h\in H$ we have that
\begin{equation}\label{eq:oneless}
\begin{split}
[g,h]^{n+1}&= [g,h][g,h]^n=ghg^{-1}h^{-1}[g,h]^n=ghg^{-1}[g,h]^nh^{-1}=\\
&=[g,h^2]h[g,h]^{n-1}h^{-1}=[g,h^2][hgh^{-1},h].
\end{split}
\end{equation}
Moreover, it is easy to see that the normal core of $C_G(H)$ is included into $C_G(\sigma(H))$, for every inner automorphism $\sigma$ of $G$. Therefore more in general the following holds: for every $h\in H$, $g\in G$ and $\sigma$ inner automorphism of $G$:
\begin{equation}
\label{eq:oneless}[g,\sigma(h)]^{n+1}=[g,\sigma(h)^2][\sigma(h)g\sigma(h)^{-1},\sigma(h)].
\end{equation}
As $C_G(H)$ has finite index, there are finitely many commutators $[g,h]$, actually there are at most $n^2$ commutators. Furthermore, there are at most $n^3$ commutators of the form $[g,\sigma(h)]$ with $h\in K$ and $\sigma$ an inner automorphism of $G$. By the previous lemma, $[g,\sigma(h)]\in (G,H)$, hence $(G,H)$ is spanned by the $[g,\sigma(h)]$.

 Then we claim that every element in $(G,H)$ can be written as a product of at most $n^4$ commutators of the form $[g,\sigma(h)]$: indeed assume that $c=c_1\cdot\dotsc\cdot c_r$ with $r>n^4$, then there is a commutator $\dot{c}$ occurring at least $n+1$ times, assume that $\dot{c}=c_i$, then 
\[c=\dot{c}\dot{c}^{-1}c_1\dot{c}\dot{c}^{-1}c_2\dot{c}\dotsc\dot{c}^{-1}c_{i-1}\dot{c}c_{i+1}\dotsc c_r.\]
As the conjugate of a commutator of the form $[g,\sigma(h)]$ for some  inner automorphism $\sigma$ is again a commutator of the same form, we can assume that $c=\dot{c}^{n+1} c_{n+2}\dotsc c_r$ but by \eqref{eq:oneless} this means that we can write it as a product of $r-1$ commutators of the form $[g,\sigma(h)]$ and by induction we have that we can always assume that $r\leq n^4$.
\end{proof}

\section{Commutators and normal hulls}

The following is a classical result:

\begin{proposition}[{\cite[$VI_B$,§7 Prop.~7.1]{SGA3} or \cite[II,§5,Prop.~4.9]{DemGabr}}] \label{comm:conn}Let $G$ be a group scheme locally of finite type. Let $H, K$ two smooth subgroup schemes, with $H$ irreducible and $K$ of finite type. Then there exists a unique smooth irreducible subscheme of $G$ such that for any $k$-algebra $R$,
\[(H,K)(R)=\{ g\in G(R)\mid \exists R\to R' \text{ fppf },g\in (H(R'),K(R'))\},\]
in particular for any $k'$ algebraically closed
\[(H,K)(k')=(H(k'),K(k')).\]
Moreover, there exist $n$ such that every element in  $(H,K)$ is the product of at most $n$ commutators.
\end{proposition}

Building on this result, we can prove the main result of this paper:

\begin{theorem}
Let $G$ be a smooth algebraic group scheme and $H\subset G$ be a smooth subgroup scheme, then the group functor $\tilde{H}^G(R)$ is representable by a smooth subgroup scheme of $G$, namely $H^G$.
\end{theorem}
\begin{proof} 
By Proposition~\ref{comm:conn} $(H,G_0)$ is a closed smooth subscheme of $G$ and represents the functor
\[(H,G_0)(R)=\{ g\in G(R)\mid\exists R\to R' \text{ fppf },g\in (H(R'),G_0(R'))\}.\]
It is easy to see that $H$ normalizes $(H,G_0)$, hence 
\[H\cdot (H,G_0)(R)=\{g\in G(R)\mid \exists R\to R' \text{ fppf },g\in H(R')\cdot (H,G_0)(R')\}\]
is represented by a closed subgroupscheme of $G$ which is smooth by \cite[$VI_B$,§7 Cor.~7.1.1]{SGA3}. Moreover $H(R)\subset (H\cdot (H,G_0))(R)\subset H(R)^{G(R)}$ therefore without loss of generality we can substitute $H$ with $H\cdot (H,G_0)$. The latter, though, is normalized by $G_0$: if $g_0,g_1\in G_0(R)$ and $l,h\in H(R)$
\[g_0 h[g_1,l] g_0^{-1}= \underbrace{g_0hg_0^{-1}}_{\in H\cdot (H,G_0)}\underbrace{[g_0g_1,l][l,g_0]}_{\in (H,G_0)}.\]

In particular we can assume without loss of generality that $H$ is normalized by $G_0$. We can apply the same arguments as before to $(G_0,H)$: it is smooth and irreducible, hence $(G,(G_0,H))$ is smooth and irreducible, and by Lemma~\ref{lem:norm} it is normal in $G$. In particular the group $(G_0,H)\cdot(G,(G_0,H))$ is well defined, but by Lemma~\ref{lem:norm}$(iii)$ the latter is simply $(G_0,H)^G$. Note that in particular there exists $n_1,n_2$ such that every element of $(G_0,H)^G$ is the product of $n_1$ commutators in $(G_0,H)$ and $n_2$ commutators in $(G,(G_0,H))$.

Fix now a separable closure $k^s$ of $k$ and consider the quotient 
\[\bar{G}(k^s)=G(k^s)/(G_0(k^s),H(k^s))^{G(k^s)}\]
 and let $\bar{H}(k^s)$ be the image of $H(k^s)$ in $\bar{G}(k^s)$. Then the image of $G_0(k^s)$ is in the centralizer $C_{\bar{G}(k^s)}(\bar{H}(k^s))$, in particular the latter has finite index. Hence $(\bar{G}(k^s),\bar{H}(k^s))=(G(k^s),H(k^s))/(G_0(k^s),H(k^s))^{G(k^s)})$ is finite, that is \[(G_0(k^s),H(k^s))^{G(k^s)}\subset (G(k^s),H(k^s))\] has finite index $d$. As $(G_0(k^s),H(k^s))^{G(k^s)}$ is closed, so is $(G(k^s),H(k^s))$. In particular it corresponds to the closed points of a normal subgroup of $G$, defined over $k^s$, that we denote $(G,H)$. Note that $(G,H)(k^s)=(G(k^s),H(k^s))$, hence it is stable under $\Gal(k^s/k)$, that is $(G,H)$ is defined over $k$ (see also \cite[$VI_B$,§7 Lemma~7.7]{SGA3}, noting that one only uses that $k$ is algebraically cosed to ensure that $k$-points are dense).

To show that $(G,H)$ represents the functor
\[(G,H)(R)=\{ g\in G(R)\mid \exists R\to R' \text{ fppf },g\in (G(R'),H(R'))\},\]
note that every element of $(G,H)$ is the product of at most $N=n_1+n_2+d$ commutators of the form $[g,h]$, but then we can use the methods of \cite[II,§5,Prop.~4.8]{DemGabr}: $(G\times H)^N\to (G,H)$ is dominant hence flat over some open $U\subset(G\times H)^N$. Hence
$U\times U\to (G,H)\times (G,H)\xrightarrow{m} (G,H)$ is flat and surjective. Let $R$ be any $k$-algebra and $g:\Spec R\to (G,H)$ a $R$-point. Let $X=\Spec R\times_{(G,H)}(U\times U)$, then the projection $X\to \Spec R$ is fppf, hence so is $\varphi:\Spec (\prod R_i)\to\Spec R$, where $\cup \Spec R_i$ is an open affine covering of $X$. Hence $g\circ\varphi\in (G,H)(\prod R_i)$ factors through $U\times U$, that is it it is a product of commutators in $(G(\prod R_i),H(\prod R_i))$. Conversely, $G\times H\to G$ given by $(g,h)$ to $hgh^{-1}g^{-1}$ factors through $(G,H)$, giving the reverse inclusion.

By Lemma~\ref{lem:norm} $(G,H)$ is normal in $G$ and $(G,H)\cdot H$ is a smooth subscheme of $G$ representing $\tilde{H}^G$, hence by minimality $H^G=(G,H)\cdot H$.
\end{proof}

\begin{corollary}
Let $G$ be an algebraic group scheme and let $H\subset G$ be a smooth subgroup scheme, then the functor
\[(H,G)(R)=\{ g\in G(R)\mid \exists R\to R' \text{ fppf },g\in (H(R'),G(R'))\}\]
is representable by a closed smooth subscheme of $G$.
\end{corollary}
\begin{proof}
It suffice to note that for every $R$ the inclusion $(H(R),G(R))\subset (H(R)\cdot (G_0(R),H(R)),G(R))$ is actually an equality: $H\cdot (G_0,H)$ is generated by all conjugates of $H$ under $G_0$, in particular it suffices to show that $[g_0hg_0^{-1},g]\in (H(R),G(R))$ for every $g_0\in G_0(R)$, $h\in H(R)$ and $g\in G(R)$. But by Lemma\ref{lem:norm} $(H(R),G(R))$ is normal in $G(R)$, in particular $[g_0hg_0^{-1},g]\in (H(R),G(R))$ if and only if $g_0^{-1}[g_0hg_0^{-1},g]g_0=[h,g_0^{-1}gg_0]\in (H(R),G(R))$. We showed in the proof of the main theorem that 
\[R\mapsto \{g\in G(R)\mid\exists R\to R' \text{ fppf },g\in H(R')\cdot (G_0(R'),H(R')),G(R'))\}\] is representable, hence so must be $(H,G)(R)$.
\end{proof}

\addcontentsline{toc}{section}{\refname}
\printbibliography
\end{document}